\documentclass{amsart}
\usepackage{amsmath}
\usepackage{amsfonts}
\usepackage{amssymb}

\providecommand{\U}[1]{\protect\rule{.1in}{.1in}}
\makeatletter
\providecommand{\bigsqcap}{\mathop{\mathpalette\@updown\bigsqcup}}
\newcommand*{\@updown}
[2]{\rotatebox[origin=c]{180}{$\m@th#1#2$}}
\makeatother
\newtheorem{theorem}{Theorem}
\theoremstyle{plain}

\newtheorem{definition}{Definition}
\newtheorem{example}{Example}

\newtheorem{proposition}{Proposition}
\newtheorem{remark}{Remark}

\numberwithin{equation}{section}

\begin{document}
\title{Comparing Diagonals on the Associahedra}
\author{Samson Saneblidze}
\author{Ronald Umble}
\address{A. Razmadze Mathematical Institute, I. Javakhishvili Tbilisi State
University 2, Merab Aleksidze II Lane, 0193 Tbilisi, Georgia}
\email{sane@rmi.ge}
\address{Department of Mathematics, Millersville University of Pennsylvania,
Millersville, PA 17551, USA}
\email{ron.umble@millersville.edu}
\date{August 22, 2022; revised February 23, 2024}

\begin{abstract}
We prove that the formula for the diagonal approximation $\Delta_{K}$ on J.
Stasheff's $n$-dimensional associahedron $K_{n+2}$ derived by the current
authors in \cite{SU} agrees with the \textquotedblleft magical
formula\textquotedblright\ for the diagonal approximation $%
\Delta_{K}^{\prime }$ derived by Markl and Shnider in \cite{MS}, by J.-L.
Loday in \cite{JL}, and more recently by Masuda, Thomas, Tonks, and Vallette
in \cite{MTTV}.
\end{abstract}

\maketitle


\vspace{.1in}

\begin{center}
\textit{Dedicated to the memory of Jean-Louis Loday}
\end{center}

\bigskip

\section{Introduction}

Recently there has been renewed interest in explicit combinatorial diagonal
approximations on J. Stasheff's $n$-dimensional associahedron $K_{n+2}$ \cite%
{St}. Markl and Shnider (M-S) in \cite{MS}, J.-L. Loday in \cite{JL}, and
more recently Masuda, Thomas, Tonks, and Vallette (MTTV) in \cite{MTTV}
constructed a diagonal $\Delta_{K}^{\prime}$ on $K_{n+2}$ whose components
are \textquotedblleft matching pairs\textquotedblright\ of faces, which in
the words of Jean-Louis Loday, are \textquotedblleft pairs of cells of
matching dimensions and comparable under the Tamari
order.\textquotedblright\ By definition, every component of the
combinatorial diagonal $\Delta_{K}$ on $K_{n+2}$ constructed by the current
authors (S-U) in \cite{SU} is a matching pair. In this paper we prove that
every matching pair is a component of $\Delta_{K}.$ Thus the S-U formula for 
$\Delta_{K}$ and the \textquotedblleft magical formula\textquotedblright\
for $\Delta_{K}^{\prime}$ agree (see Definitions \ref{S-U} and \ref{magical}%
).

Historically, S-U were the first to derive a cellular
combinatorial/differential graded formula for $\Delta_{K},$ M-S were the
first to prove the magical formula for $\Delta_{K}^{\prime},$ and MTTV were
the first to construct a point-set topological diagonal map, which descends
to the magical formula at the cellular level.

Using the geometric methods of MTTV, Laplante-Anfossi created a general
framework for studying diagonals on any polytope in \cite{L}. In this
framework, a choice of diagonal on the $n$-dimensional permutahedron $%
P_{n+1} $ is given by a choice of chambers in its fundamental hyperplane
arrangement (\cite{L}, Def. 1.18). While the specific diagonal $\Delta
_{P}^{\prime }$ on $P_{n+1}$ studied in \cite{L} differs from the S-U
diagonal $\Delta _{P}$, the diagonal $\Delta _{K}^{\prime }$ on $K_{n+2}$
induced by $\Delta _{P}^{\prime }$ agrees with $\Delta _{K}.$ \bigskip

\noindent\textbf{Acknowledgments. } We wish to thank Bruno Vallette for
sharing his perspective on the history of combinatorial diagonals on $K_{n}$%
, and Guillaume Laplante-Anfossi and our anonymous referee for their helpful
editorial suggestions.

\section{Diagonals Induced by $\Delta_P$}

Let $S_{n}$ be the symmetric group on the finite set $\underline{n}=\left\{
1,2,\ldots ,n\right\} .$ The permutahedron $P_{n}$ is the convex hull of $n!$
vertices $\{\left( \sigma(1),\ldots,\sigma(n)\right):\sigma\in S_n\}\subset 
\mathbb{R}^{n}$. As a cellular complex, $P_{n}$ is an $\left( n-1\right) $%
-dimensional convex polytope whose $\left( n-p\right) $-faces are indexed by
(ordered) partitions $A_{1}|\cdots|A_{p}$ of $\underline{n}$, $1\leq p\leq n$%
. Denoting the set of ordered partitions of $\underline{n}$ by $P\left(
n\right) , $ the faces of $P_{n}$ are identified with elements of $P\left(
n\right) $ in the standard way.

Let $X$ be an $n$-dimensional polytope that admits a (surjective) cellular
projection map $p:P_{n+1}\rightarrow X$ and a realization as a subdivision
of the $n$-cube $I^{n}$, i.e., for $0\leq k\leq n$, each $k$-cell ($k$%
-subcube) of $I^n$ is a union of $k$-cells of $X$, any two of which
intersect along their boundaries.

For example, $X=P_{n}$ can be realized as a subdivision of $I^{n-1}$
inductively as follows: Identify $P_{1}$ with $1\in P\left( 1\right) .$ If $%
P_{n-1}$ has been constructed and $a=A_{1}|\cdots|A_{p}\in P\left(
n-1\right) $ is a face, let $a_{0}=0,$ $a_{j}=\#\left(
A_{p-j+1}\cup\cdots\cup A_{p}\right) $ for $0<j<p,\ a_{p}=\infty$, and
define $\frac{1}{2^{\infty}}:=0$. Let\emph{\ }$I\left( a\right) :=I_{1}\cup
I_{2}\cup\cdots\cup I_{p},$ where $I_{j}:=[1-\frac{1}{2^{a_{j-1}}},1-\frac{1%
}{2^{a_{j}}}];$ then $P_{n}=\bigcup\nolimits_{a\in P\left( n-1\right)
}a\times I\left( a\right) ,$ where the identification of faces with
partitions is given by 
\begin{equation*}
\begin{tabular}{c|rr}
$_{\mathstrut}^{\mathstrut}$\textbf{Face of }$a\times I\left( a\right) $ & 
\textbf{Partition in }$P\left( n\right) $ &  \\ \hline
$_{\mathstrut}^{\mathstrut}a\times0$ & \multicolumn{1}{|c}{$A_{1}|\cdots
|A_{p}|n$} & \multicolumn{1}{l}{} \\ 
$_{\mathstrut}^{\mathstrut}a\times(I_{j}\cap I_{j+1})$ & \multicolumn{1}{|c}{%
$A_{1}|\cdots|A_{p-j}|n|A_{p-j+1}|\cdots|A_{p},$} & \multicolumn{1}{l}{$%
1\leq j\leq p-1$} \\ 
$_{\mathstrut}^{\mathstrut}a\times1$ & \multicolumn{1}{|c}{$%
n|A_{1}|\cdots|A_{p},$} & \multicolumn{1}{l}{} \\ 
$_{\mathstrut}^{\mathstrut}a\times I_{j}$ & \multicolumn{1}{|c}{$%
A_{1}|\cdots|A_{p-j+1}\cup n|\cdots|A_{p},$} & \multicolumn{1}{l}{$1\leq
j\leq p$}%
\end{tabular}
\ \ \ \ 
\end{equation*}

\noindent(see Figures 1 and 2). We refer to a vertex common to $P_{n}$ and $%
I^{n-1}$ as a \emph{cubical vertex}. Thus $a$ is a cubical\ vertex of $P_{n}$
if and only if $a|n$ and $n|a$ are cubical\ vertices of $P_{n+1}.$ Indeed, a
cubical vertex has the form $a=a_1|\cdots|a_{i-1}|1|a_{i+1}|\cdots|a_{n} $,
where $a_1>\cdots >a_{i-1}$ and $a_{i+1}<\cdots <a_{n}.$

We begin with a review of the diagonal $\Delta_{P}$ and the diagonal $%
\Delta_{X}$ induced by the projection $p;$ then $\Delta_{K}$ is obtained by
setting $X=K_{n+2}$. Whereas the vertices of $P_{n+1}$ are identified with
the permutations in $S_{n+1}$, the \emph{weak order} on $S_{n+1}$ given by $%
\cdots|x_{i}|x_{i+1}|\cdots <\cdots|x_{i+1}|x_{i}|\cdots$ if $x_{i}<x_{i+1}$
extends to a partial order (p-o) and the associated Hasse diagram orients
the $1$-skeleton of $P_{n+1}$ \cite{CP}. Denote the minimal and maximal
vertices of a face $e$ of $P_{n+1}$ by $\min e$ and $\max e$, respectively,
and define $e\leq e^{\prime}$ if there exists an oriented edge-path in $%
P_{n+1}$ from $\max e$ to $\min e^{\prime}$. Then $p$ induces a p-o on the
cells of $X$. For example, when the faces of $P_{n+1}$ are indexed by planar
leveled trees (PLTs) with $n+2$ leaves and the faces of $K_{n+2}$ are
indexed by planar rooted trees (PRTs) with $n+2$ leaves (without levels),
Tonks' projection $p=\theta$ given by forgetting levels \cite{Tonks} induces
the \emph{Tamari order} on the faces $\{\theta(T_{i})\}$ of $K_{n+2}$ given
by $\theta(T_{i})\leq\theta(T_{j})$ if $T_{i}\leq T_{j}$. In particular, the
vertices of $K_{n+1}$ form a subset of the vertices $P_{n}$ and the Tamari
order restricted to this subset agrees with the weak order.

Let $e$ be a cell of $X$ and let $\left\vert e\right\vert $ denote its
dimension$.$ A $k$\emph{-subdivision cube of }$e$ is a set of faces of $e$
whose union is a $k$-subcube of $I^{n}$ for some $k\leq n.$ For example,
when $e$ is the top dimensional cell of $P_{4},$ the facets in $\left\{
2|134,24|13\right\} $ and $\{2|134,24|13,23|14,234|1\}$ form $2$-subdivision
cubes of $e$, but any three in the latter do not (see Figure 2). Denote the
set of vertices of $e$ by $\mathcal{V}_{e}$ (when $e=X$ we suppress the
subscript $e$). Given a vertex $v\in\mathcal{V}_{e},$ let $I_{v,1}^{k_{1}}$
and $I_{v,2}^{k_{2}}$ be $k_{i}$-subdivision cubes of $e$ such that $\max
I_{v,1}^{k_{1}}=\min I_{v,2}^{k_{2}}=v$ and $k_{1}+k_{2}=|e|;$ then $\left(
I_{v,1}^{k_{1}},I_{v,2}^{k_{2}}\right) $ is a \emph{pair of} $\left(
k_{1},k_{2}\right) $-\emph{subdivision cubes of }$e.$ Denote the set of all
such pairs by $e_{v}$ and let $(\mathbf{I}_{v,1}^{k_{1}},\mathbf{I}%
_{v,2}^{k_{2}})_{e}$ denote its unique maximal element; then $\left(
I_{v,3}^{k_{3}},I_{v,4}^{k_{4}}\right) \subseteq(\mathbf{I}_{v,1}^{k_{1}},%
\mathbf{I}_{v,2}^{k_{2}})_{e}$ for all $(I_{v,3}^{k_{3}},I_{v,4}^{k_{4}})\in
e_{v}.$ For example, when $e$ is the top dimensional cell of $P_{4}$ and $%
v=4|2|3|1,$ we have $(\mathbf{I}_{v,1}^{2},\mathbf{I}_{v,2}^{1})_{e}=\left(
\left\{ 2|134,24|13\right\} ,\left\{ 4|23|1\right\} \right) $. For an
explicit description of $\left( \mathbf{I}_{v,1}^{k_{1}},\mathbf{I}%
_{v,2}^{k_{2}}\right) _{\!e}$ when $e\subseteq P_{n}$ see (\ref{maxpair1})
below. \vspace{.2in}

\begin{center}
\setlength{\unitlength}{0.0004in}%
\begin{picture}
(2975,2685)(3126,-2038) \thicklines \put(3601,239){\line(
1,0){1800}} \put(5401,239){\line( 0,-1){1800}}
\put(5401,-1561){\line(-1, 0){1800}} \put(3601,-1561){\line(
0,1){1800}} \put(3601,239){\makebox(0,0){$\bullet$}}
\put(3601,-661){\makebox(0,0){$\bullet$}}
\put(3601,-1561){\makebox(0,0){$\bullet$}}
\put(5401,239){\makebox(0,0){$\bullet$}}
\put(5401,-661){\makebox(0,0){$\bullet$}}
\put(5401,-1561){\makebox(0,0){$\bullet$}}
\put(4500,-680){\makebox(0,0){$123$}}
\put(2980,-1861){\makebox(0,0){$1|2|3$}}
\put(2980,-699){\makebox(0,0){$1|3|2$}}
\put(2980,464){\makebox(0,0){$3|1|2$}}
\put(6000,-1861){\makebox(0,0){$2|1|3$}}
\put(6000,-699){\makebox(0,0){$2|3|1$}}
\put(6000,464){\makebox(0,0){$3|2|1$}}
\put(3040,-1260){\makebox(0,0){$1|23$}}
\put(4550,530){\makebox(0,0){$3|12$}}
\put(3040,-111){\makebox(0,0){$13|2$}}
\put(5960,-111){\makebox(0,0){$23|1$}}
\put(5960,-1260){\makebox(0,0){$2|13$}}
\put(4550,-1890){\makebox(0,0){$12|3$}}
\end{picture}\vspace{0.1in}

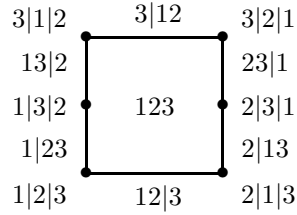
Figure 1: $P_{3}$ as a subdivision of $P_{2}\times I.$ \vspace{.1in}

\setlength{\unitlength}{0.007in} 
\begin{picture}
	(500,-500) \thicklines
	\put(0,-120){\line( 0,-1){120}} \put(120,0){\line( 0,-1){360}}
	\put(180,-120){\line( 0,-1){120}} \put(240,0){\line( 0,-1){360}}
	\put(360,-120){\line( 0,-1){120}} \put(420,-120){\line(
		0,-1){120}} \put(480,-120){\line( 0,-1){120}}
	\put(120,0){\line( 1,0){120}} \put(0,-120){\line( 1,0){480}} \put
	(120,-150){\line( 1,0){60}} \put(180,-180){\line( 1,0){60}} \put
	(240,-150){\line( 1,0){120}} \put(420,-150){\line( 1,0){60}} \put
	(0,-180){\line( 1,0){120}} \put(360,-180){\line( 1,0){60}}
	\put(0,-240){\line( 1,0){480}} \put(120,-360){\line( 1,0){120}}
	\put(120,0){\makebox(0,0){$\bullet$}}
	\put(180,0){\makebox(0,0){$\bullet$}}
	\put(240,0){\makebox(0,0){$\bullet$}}
	\put(0,-120){\makebox(0,0){$\bullet$}}
	\put(120,-120){\makebox(0,0){$\bullet$}} \put(180,-120){\makebox
		(0,0){$\bullet$}} \put(240,-120){\makebox(0,0){$\bullet$}} \put
	(360,-120){\makebox(0,0){$\bullet$}}
	\put(420,-120){\makebox(0,0){$\bullet$}}
	\put(480,-120){\makebox(0,0){$\bullet$}} \put(120,-150){\makebox
		(0,0){$\bullet$}} \put(180,-150){\makebox(0,0){$\bullet$}} \put
	(240,-150){\makebox(0,0){$\bullet$}}
	\put(360,-150){\makebox(0,0){$\bullet$}}
	\put(0,-180){\makebox(0,0){$\bullet$}}
	\put(182.25,-180){\makebox(0,0){$\bullet$ }}
	\put(240,-180){\makebox(0,0){$\bullet$}} \put(420,-180){\makebox
		(0,0){$\bullet$}} \put(480,-180){\makebox(0,0){$\bullet$}} \put
	(0,-150){\makebox(0,0){$\bullet$}}
	\put(120,-240){\makebox(0,0){$\bullet$}}
	\put(360,-240){\makebox(0,0){$\bullet$}} \put(420,-240){\makebox
		(0,0){$\bullet$}} \put(480,-150){\makebox(0,0){$\bullet$}} \put
	(0,-240){\makebox(0,0){$\bullet$}}
	\put(120,-180){\makebox(0,0){$\bullet$}}
	\put(180,-240){\makebox(0,0){$\bullet$}} \put(240,-240){\makebox
		(0,0){$\bullet$}} \put(360,-180){\makebox(0,0){$\bullet$}} \put
	(420,-150){\makebox(0,0){$\bullet$}}
	\put(480,-240){\makebox(0,0){$\bullet$}}
	\put(120,-360){\makebox(0,0){$\bullet$}} \put(180,-360){\makebox
		(0,0){$\bullet$}} \put(240,-360){\makebox(0,0){$\bullet$}}
	\put(55,-152){\makebox(0,0){$124|3$}}
	\put(55,-210){\makebox(0,0){$12|34$}}
	\put(150,-134){\makebox(0,0){$24|13$}}
	\put(150,-195){\makebox(0,0){$2|134$}}
	\put(177,-55){\makebox(0,0){$4|123$}}
	\put(177,-295){\makebox(0,0){$123|4$}}
	\put(210,-152){\makebox(0,0){$234|1$}}
	\put(210,-210){\makebox(0,0){$23|14$}}
	\put(300,-134){\makebox(0,0){$34|12$}}
	\put(300,-195){\makebox(0,0){$3|124$}}
	\put(390,-152){\makebox(0,0){$134|2$}}
	\put(390,-200){\makebox(0,0){$13|24$}}
	\put(450,-134){\makebox(0,0){$14|23$}}
	\put(450,-194){\makebox(0,0){$1|234$}}
\end{picture}
\vspace*{2.7in}

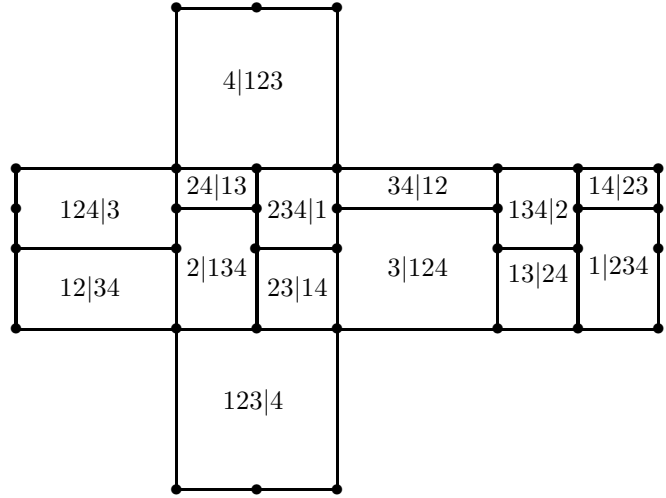
Figure 2: The facets of $P_{4}$ as a subdivision of $I^{3}$.\bigskip
\end{center}


If in addition, the cellular projection $p:P_{n+1}\rightarrow X$ preserves
maximal pairs of $\left( k_{1},k_{2}\right) $-subdivision cubes, i.e., for
every cell $e$ of $P_{n+1}$ we have 
\begin{equation*}
p\left( \mathbf{I}_{v,1}^{k_{1}},\mathbf{I}_{v,2}^{k_{2}}\right) _{e}=\left( 
\mathbf{I}_{p(v),1}^{k_{1}},\mathbf{I}_{p(v),2}^{k_{2}}\right) _{p(e)},
\end{equation*}
the components of the induced diagonal $\Delta_{X}$ on a cell $f\subseteq X$
form the set of product cells 
\begin{equation}
\Delta_{X}(f):=\bigcup_{\substack{ \left( e^{k_{1}},\,e^{k_{2}}\right)
\in\left( \mathbf{I}_{v,1}^{k_{1}},\mathbf{I}_{v,2}^{k_{2}}\right) _{\!f} 
\\ v\in\mathcal{V}_{f}}}\{e^{k_{1}}\times e^{k_{2}}\}.  \label{thepair}
\end{equation}

In particular, $p=\theta$ preserves maximal pairs of $(k_{1},k_{2})$%
-subdivision cubes and $\Delta_{K}(e)$ is given by setting $X=K_{n+2}$ (see (%
\ref{S-U}) below). Note that $\left( e^{k_{1}},e^{k_{2}}\right) \in\left( 
\mathbf{I}_{v,1}^{k_{1}},\mathbf{I}_{v,2}^{k_{2}}\right) _{\!X}$ implies $%
e^{k_{1}}\leq e^{k_{2}}.$ Thus $e^{k_{1}}\times e^{k_{2}}$ is a
\textquotedblleft matching pair\textquotedblright\ in the sense of MTTV (see
Definition \ref{mp}). Furthermore, since $f=p\left( e\right) $ for some 
\begin{equation*}
e=P_{n_{1}}\times\cdots\times P_{n_{s}} \text{ and }p\left( e\right)
=p(P_{n_{1}})\times\cdots\times p(P_{n_{s}}),
\end{equation*}
the diagonal $\Delta_{X}(f)$ is automatically the comultiplicative extension
of its values on the factors of $f$, i.e.,

\begin{equation*}
\Delta_{X}(f)=\Delta_{X}(p(P_{n_{1}}))\times\cdots\times\Delta
_{X}(p(P_{n_{s}})).
\end{equation*}

The subset $\mathcal{V}_{e}\subseteq S_{n}$ determines the components of $%
\Delta _{P}(e)$ in the following way: Let $\sigma =x_{1}|\cdots |x_{n}\in 
\mathcal{V}_{e}.$ Reading $\sigma $ from left-to-right and from
right-to-left, construct the partitions $\overleftarrow{\sigma }_{1}|\cdots |%
\overleftarrow{\sigma }_{p}$ and $\overrightarrow{\sigma }_{q}|\cdots |%
\overrightarrow{\sigma }_{1}$ of maximal decreasing subsets and form the 
\emph{Strong Complementary Pair }(SCP) 
\begin{equation*}
a_{\sigma }\times b_{\sigma }:=\overleftarrow{\sigma }_{1}|\cdots |%
\overleftarrow{\sigma }_{p}\times \overrightarrow{\sigma }_{q}|\cdots |%
\overrightarrow{\sigma }_{1}\in P(n)\times P(n).
\end{equation*}%
Then 
\begin{equation*}
\sigma =\max a_{\sigma }=\min b_{\sigma },\ \min \overleftarrow{\sigma }%
_{j}<\max \overleftarrow{\sigma }_{j+1}\text{ for all }j<p,\text{ and }
\end{equation*}%
\begin{equation*}
\min \overrightarrow{\sigma }_{i}<\max \overrightarrow{\sigma }_{i+1}\text{
for all }i<q.
\end{equation*}%
Thus, for $\sigma =2|1|3|5|4$ we have $\overleftarrow{\sigma }_{1}|%
\overleftarrow{\sigma }_{2}|\overleftarrow{\sigma }_{3}=21|3|54$ and $%
\overrightarrow{\sigma }_{3}|\overrightarrow{\sigma }_{2}|\overrightarrow{%
\sigma }_{1}=2|135|4$ so that $a_{\sigma }\times b_{\sigma }=21|3|54\times
2|135|4$.

Let $a=A_{1}|\cdots|A_{p}\in P(n)$. For $1\leq j<p$, let $%
M_{j}\subseteq(A_{j}\smallsetminus\{\min A_{j}\})$ such that $\min
M_{j}>\max A_{j+1}$ when $M_j\neq \varnothing$. Define the \emph{right-shift 
}$M_{j}$ \emph{action} 
\begin{equation*}
R_{M_{j}}(a):=\left\{ 
\begin{array}{cl}
A_{1}|\cdots|A_{j}\smallsetminus M_{j}|A_{j+1}\cup M_{j}|\cdots|A_{k}, & 
M_{j}\neq\varnothing \\ 
a, & M_{j}=\varnothing.%
\end{array}
\right.
\end{equation*}
Let \textbf{M}$:=(M_{1},M_{2} ,\ldots,M_{p-1})$ and denote the composition ${%
R}_{M_{p-1}}\cdots R_{M_{2} }R_{M_{1}}(a)$ by $R_{\mathbf{M}}\left(
a\right). $

Dually, let $b=B_{q}|\cdots|B_{1}\in P(n).$ For $1\leq i<q,$ let $%
N_{i}\subseteq(B_{i}\smallsetminus\{\min B_{i}\})$ such that $\min
N_{i}>\max B_{i+1}$ when $N_{i}\neq\varnothing.$ Define the \emph{left-shift 
}$N_{i} $\emph{\ action} 
\begin{equation*}
L_{N_{i}}(b):=\left\{ 
\begin{array}{cl}
B_{q}|\cdots|B_{i+1}\cup N_{i}|B_{i}\smallsetminus N_{i}|\cdots|B_{1}, & 
N_{i}\neq\varnothing \\ 
b, & N_{i}=\varnothing.%
\end{array}
\right.
\end{equation*}
Let \textbf{N}$:=(N_{1},N_{2} ,\ldots,N_{q-1})$ and denote the composition $%
L_{N_{q-1}}\cdots L_{N_{2}} L_{N_{1}}(b)$ by $L_{\mathbf{N}}\left( b\right)
. $

Now given $\sigma \in \mathcal{V}_e$ and the SCP $a_\sigma\times b_\sigma,$
the pair ${R}_{\mathbf{M}}(a_{\sigma})\times L_{\mathbf{N}}\left(
b_{\sigma}\right) $ is a \emph{Complementary Pair} (CP) \emph{on} $%
a_\sigma\times b_\sigma$. Define 
\begin{equation*}
A_{\sigma}\times B_{\sigma}:=\bigcup\limits_{\mathbf{M,N}}\left\{ {R}_{%
\mathbf{M}}(a_{\sigma})\times L_{\mathbf{N}}\left( b_{\sigma}\right) \right\}
\end{equation*}
and 
\begin{equation}
\Delta_{P}(e):=\bigcup_{\sigma\in\mathcal{V}_{e}}A_{\sigma}\times B_{\sigma}.
\label{Delta_P}
\end{equation}

\begin{example}
On the top dimensional cell $e^{2}\subseteq P_{3}$, $\Delta_{P}(e^{2})$ is
the union of 
\begin{equation*}
\hspace{-0.1in} 
\begin{array}{ll}
A_{1|2|3}\times B_{1|2|3}=\left\{ 1|2|3\times123\right\} , & A_{1|3|2}\times
B_{1|3|2}=\left\{ 1|32\times13|2\right\} , \\ 
A_{2|1|3}\times B_{2|1|3}=\left\{ 21|3\times2|13,\text{ }21|3\times
23|1\right\} , & A_{2|3|1}\times B_{2|3|1}=\{2|31\times23|1\}, \\ 
A_{3|1|2}\times B_{3|1|2}=\{31|2\times3|12,\text{ }1|32\times3|12\}, & 
A_{3|2|1}\times B_{3|2|1}=\{321\times3|2|1\}.%
\end{array}%
\end{equation*}
\end{example}

\begin{remark}
Note that the matrix representation of a CP introduced in \cite{SU}
conveniently organizes and systematizes the combinatorial calculation of $%
\Delta_P$. An SCP is represented by a \emph{step matrix} and a general CP is
represented by a \emph{derived matrix}, given by left-shift and down-shift
actions on a step matrix.
\end{remark}

When $X=P_{n+1}$, Formulas (\ref{thepair}) and (\ref{Delta_P}) are
equivalent. The maximal $\left( k_{1},k_{2}\right) $-subdivision pair with
respect to a vertex $\sigma$ of $P_{n+1}$ is\ 
\begin{equation}
\left( \mathbf{I}_{\sigma,1}^{k_{1}},\mathbf{I}_{\sigma,2}^{k_{2}}\right)
=\left( \bigcup_{e_{1}\in A_{\sigma}}e_{1},\bigcup_{e_{2}\in
B_{\sigma}}e_{2}\right) .  \label{maxpair1}
\end{equation}

\begin{definition}
A positive dimensional face $e$ of $P_{n}$ is \textbf{non-degenerate }if $%
|\theta(e)|=|e|$. A positive dimensional partition $a=A_{1}|\cdots|A_{p}\in
P(n)$ is \textbf{degenerate }if for some $j$ and some $k>0$, there exist $%
x,z\in A_{j}$ and $y\in A_{j+k}$ such that $x<y<z;$ otherwise $a$ is \textbf{%
non-degenerate}. A CP $\alpha\times\beta$ is \textbf{non-degenerate} if $%
\alpha$ and $\beta$ are non-degenerate.
\end{definition}

\noindent Define $\Delta_{K}(K_{n+1})=\Delta_{K}(\theta(P_{n})):=(\theta
\times\theta)\Delta_{P}(P_{n});$ then%
\begin{equation}
\Delta_{K}(e^{n-1})=\bigcup_{\substack{ \text{non-degenerate CPs }  \\ %
\alpha\times\beta\in A_{\sigma}\times B_{\sigma}  \\ \sigma\in S_{n}}}%
\{\theta\left( \alpha\right) \times\theta\left( \beta\right) \}.  \label{S-U}
\end{equation}

\section{Agreement of $\Delta_K$ and $\Delta^{\prime }_{K}$}

\begin{definition}
\label{mp} A pair of faces $a\times b\subseteq K_{n+1}\times K_{n+1}$ is a 
\textbf{Matching Pair }(MP) if $a\leq b$ and $\left\vert a\right\vert
+\left\vert b\right\vert =n-1.$
\end{definition}

\noindent The \textquotedblleft magical formula\textquotedblright\ derived
in \cite{MS} and \cite{MTTV} is 
\begin{equation}
\Delta_{K}^{\prime}\left( e^{n-1}\right) =\bigcup_{\substack{ \text{MPs of
faces}  \\ a\times b\subseteq K_{n+1}\times K_{n+1}}}\{a\times b\}.
\label{magical}
\end{equation}

Tonks' projection $\theta$ sends every non-degenerate CP to an MP. The
converse is our main result: \emph{Every MP is the image of a unique
non-degenerate CP under }$\theta $\emph{; thus }$\Delta_{K}^{\prime}$\emph{\
and }$\Delta_{K}$\emph{\ agree}. Our proof of this fact views $P_{n}$ as a
subdivision of $K_{n+1}$.

\begin{definition}
Let $0\leq k<n.$ An \textbf{associahedral }$k$-\textbf{cell of }$P_{n}$ is a 
$k$-cell of $K_{n+1}$. A \textbf{subdivision }$k$-\textbf{cell of} $P_{n}$
is a $k$-cell of some associahedral $k$-cell of $P_{n}$. The \textbf{maximal
(}respt.\textbf{\ minimal)} \textbf{subdivision }$k$\textbf{-cell} of an
associahedral $k$-cell $a,$ denoted by $a_{\max}$ (respt. $a_{\min}$),
satisfies $\max a_{\max}=\max a$ (respt. $\min a_{\min}=\min a).$ A \textbf{%
non-degenerate vertex of} $P_{n}$ is an associahedral vertex.
\end{definition}

\noindent Thus a subdivision $k$-cell of $P_{n}$ has the form $%
A_{1}|\cdots|A_{n-k}.$ In fact, a vertex $v$ of $P_{n}$ is associahedral if
and only if the $\left( n-q\right) $-cell $\overrightarrow{v}_{q}|\cdots |%
\overrightarrow{v}_{1}$ is non-degenerate, in which case $\min \overset{%
\rightarrow}{v}_{q}>\cdots>\min\overset{\rightarrow}{v}_{1}.$ If $k>0,$ an
associahedral $k$-cell $a$ is a subdivision $k$-cell if and only if $%
a=a_{\min}.$

\begin{proposition}
\label{minimal}If $a$ is an associahedral $k$-cell and $u$ is a subdivision $%
k$-cell of $a,$ then

\begin{enumerate}
\item[\textit{(i)}] $a_{\min}$ is non-degenerate.

\item[\textit{(ii)}] If $u\neq a_{\min},$ then $u$ is degenerate and $u=L_{%
\mathbf{N}}\left( a_{\min}\right) $ for some $\mathbf{N.}$

\item[\textit{(iii)}] $a_{\min}=R_{\mathbf{M}}\left( a_{\max}\right) $ for
some $\mathbf{M}$.
\end{enumerate}
\end{proposition}

\begin{proof}
Set $p=n-k$ and consider an associahedral $k$-cell $a$ of $P_{n}$. If $a$ is
also a subdivision $k$-cell, then $a=a_{\min}=\theta\left( a\right) $ is
non-degenerate and $\mathbf{M=\varnothing}.$ Otherwise, conclusions (i) and
(ii) follow from the construction of $P_{n}$ as a subdivision of $K_{n+1}.$
For part (iii), given a subdivision $k$-cell $u=A_{1}|\cdots|A_{p}$ of $a,$
let 
\begin{equation*}
N_{p}:=\{x\in A_{p}\smallsetminus\left\{ \min A_{p}\right\} :x>\max
A_{p-1}\}.
\end{equation*}
Inductively, if $1<i<p$ and $N_{i+1}$ has been constructed, let $%
A_{i}^{\prime}:=A_{i}\cup N_{i+1}$ and let 
\begin{equation*}
N_{i}:=A_{i}^{\prime }\smallsetminus\left\{ x\in
A_{i}^{\prime}\smallsetminus\left\{ \min A_{i}^{\prime}\right\} :x>\max
A_{i-1}\right\} .
\end{equation*}
Then $a_{\max}=L_{(N_{p},\ldots,N_{2})}\left( a_{\min}\right) $. Set $%
\mathbf{M}=(M_{1},...,M_{p-1}):=(N_{2},...,N_{p});$ then $a_{\min}=R_{%
\mathbf{M}}\left( a_{\max}\right) .$
\end{proof}

\begin{example}
\label{facet}Consider the associahedral facet $a=1|234\cup13|24\cup
14|23\cup134|2;$ then $a_{\min}=1|234$ is non-degenerate, $13|24=L_{\left\{
3\right\} }\left( a_{\min}\right) ,$ $14|23=L_{\left\{ 4\right\} }\left(
a_{\min}\right) ,$ and $a_{\max}=134|2=L_{\left\{ 3,4\right\} }\left(
a_{\min}\right) .$ Furthermore, $a_{\min}=1|234=R_{\left\{ 3,4\right\}
}\left( 134|2\right) .$
\end{example}

\begin{proposition}
\label{left-action}Let $v$ be an associahedral vertex of $P_{n}$ and let $a=%
\overrightarrow{v}_{q}|\cdots|\overrightarrow{v}_{1}.$ If $b$ is a
non-degenerate cell of $P_{n}$ such that $|b|=|a|$ and $\min a\leq\min b,$
then $b=L_{\mathbf{N}}(a)$ for some $\mathbf{N}.$
\end{proposition}

\begin{proof}
Let $a=A_{n-k}|\cdots|A_{1}$ and let $r_{i}=\min A_{i}.$ Since $v$ is
associahedral, it follows that $r_{n-k}>\cdots>r_{1}$. Since $\min a\leq\min
b,$ there is a product of p-o increasing transpositions $\tau:=\tau_{t}%
\cdots\tau_{2}\tau _{1}\ $such that $\tau(\min a)=\min b$ and $\tau_{i}$
preserves the inequality $r_{j}>r_{j-1}$ for $1\leq i\leq t$ and $1\leq
j\leq n-k.$ Define $\tau _{0}:=\mathbf{Id}$ and consider the vertex $%
v_{i}:=\tau_{t_{i}}\cdots \tau_{1}\tau _{0}(\min a)$ for each $1\leq
t_{i}\leq t.$ For each $i,$ there is the (possibly degenerate) cell $u_{i}:=%
\overrightarrow{v_{i}}_{q}|\cdots |\overrightarrow{v_{i}}_{1},$ where $%
q\in\left\{ n-k,n-k+1\right\} .$ Thus there is the sequence $\{
a=u_{0},u_{1},\ldots, u_{t}=b\} $ and its subsequence of $k$-cells $\left\{
a=u_{i_{0}},u_{i_{1}},...,u_{i_{s-1}},u_{i_{s}}=b\right\}.$ By construction,
for $1\leq j\leq s$, there exists $n_{j}\in\underline{n}$ such that $%
u_{i_{j}}=L_{\{n_{j}\}}(u_{i_{j-1}}).$ For $1\leq i<s,$ let 
\begin{equation*}
N_{i}=\{n_{j}\in A_{i}\cup N_{1}\cup\cdots\cup
N_{i-1}:u_{i_{j}}=L_{\{n_{j}\}}(u_{i_{j-1}})\text{ for some }j\}
\end{equation*}
and form the sequence of sets $\mathbf{N}:=(N_{s-1},...,N_{1}).$ Since $b$
is non-degenerate, the action $L_{\mathbf{N}}(a)$ is defined and $L_{\mathbf{%
N}}(a)=b$.
\end{proof}

Identify a $k$-face $F\subset K_{n+1}$ with its corresponding associahedral $%
k$-cell of $P_{n}$ and label $F$ with its minimal subdivision $k$-cell $%
F_{\min};$ then $\theta\left( F_{\min}\right) =F$ (compare Figures 2 and 3).

\begin{example}
\label{transpositions}Consider the associahedral vertex $v=5|3|1|2|4|6,$ the
associated $3$-cell $a=\overrightarrow{v}_{3}|\overrightarrow{v}_{2}|%
\overrightarrow{v}_{1}=5|3|1246$ and the non-degenerate $3$-cell $b=56|34|12.
$ Then 
\begin{equation*}
\min a=5|3|1|2|4|6<5|6|3|4|1|2=\min b,
\end{equation*}%
and there is the product of p-o increasing transpositions 
\begin{equation*}
\tau =\tau _{6}\cdots \tau _{1}:=\left( 3,6\right) \left( 4,6\right) \left(
1,6\right) \left( 2,6\right) (1,4)(2,4)
\end{equation*}%
such that 
\begin{equation*}
\{v_{1}=\tau _{1}(\min a)=5|3|1|4|2|6,v_{2}=\tau _{2}(v_{1})=5|3|4|1|2|6,\
v_{3}=\tau _{3}(v_{2})=5|3|4|1|6|2,\ 
\end{equation*}%
\begin{equation*}
v_{4}=\tau _{4}(v_{3})=5|3|4|6|1|2,\ v_{5}=\tau
_{5}(v_{4})=5|3|6|4|1|2,v_{6}=\tau _{6}(v_{5})=5|6|3|4|1|2\}.
\end{equation*}%
There is the sequence of cells 
\begin{equation*}
\{u_{0}=5|3|1246,u_{1}=5|3|14|26,u_{2}=5|34|126,u_{3}=5|34|16|2,
\end{equation*}%
\begin{equation*}
u_{4}=5|346|12,u_{5}=5|36|4|12,u_{6}=56|34|12\}
\end{equation*}%
and its subsequence of $3$-cells 
\begin{equation*}
\{u_{0}=5|3|1246,u_{2}=5|34|126,u_{4}=5|346|12,u_{6}=56|34|12\}.
\end{equation*}%
Thus 
\begin{equation*}
N_{1}=\{n_{j}\in A_{1}:u_{i_{j}}=L_{\left\{ n_{j}\right\} }\left(
u_{i_{j-1}}\right) \text{ for some }j\}=\left\{ 4,6\right\} ,\text{ and}
\end{equation*}
\begin{equation*}
N_{2}=\{n_{j}\in A_{2}\cup N_{1}:u_{i_{j}}=L_{\left\{ n_{j}\right\} }\left(
u_{i_{j-1}}\right) \text{ for some }j\}=\left\{ 6\right\} .
\end{equation*}%
Conclude that $56|34|12=L_{(\{4,6\},\left\{ 6\right\} )}(5|3|1246).$
\end{example}

\medskip

\begin{center}
\setlength{\unitlength}{0.007in}%
\begin{picture}
	(480,0) \thicklines
	\put(0,-120){\line( 0,-1){120}} \put(120,0){\line( 0,-1){360}}
	\put (180,-120){\line( 0,-1){120}} \put(240,0){\line( 0,-1){360}}
	\put (360,-120){\line( 0,-1){120}} \put(480,-120){\line(
		0,-1){120}}
	\put(120,0){\line( 1,0){120}} \put(0,-120){\line( 1,0){480}} \put
	(180,-180){\line( 1,0){60}} \put(240,-150){\line( 1,0){120}} \put
	(0,-240){\line( 1,0){480}} \put(120,-360){\line( 1,0){120}}
	\put(120,0){\makebox(0,0){$\bullet$}}
	\put(240,0){\makebox(0,0){$\bullet$}}
	\put(0,-120){\makebox(0,0){$\bullet$}}
	\put(120,-120){\makebox(0,0){$\bullet$}}
	\put(180,-120){\makebox(0,0){$\bullet$}} \put(240,-120){\makebox
		(0,0){$\bullet$}} \put(360,-120){\makebox(0,0){$\bullet$}} \put
	(480,-120){\makebox(0,0){$\bullet$}}
	\put(240,-150){\makebox(0,0){$\bullet$}}
	\put(360,-150){\makebox(0,0){$\bullet$}} \put(180,-180){\makebox
		(0,0){$\bullet$}} \put(240,-180){\makebox(0,0){$\bullet$}} \put
	(120,-240){\makebox(0,0){$\bullet$}}
	\put(360,-240){\makebox(0,0){$\bullet$}}
	\put(0,-240){\makebox(0,0){$\bullet$}}
	\put(180,-240){\makebox(0,0){$\bullet$}}
	\put(240,-240){\makebox(0,0){$\bullet$}} \put(480,-240){\makebox
		(0,0){$\bullet$}} \put(120,-360){\makebox(0,0){$\bullet$}} \put
	(240,-360){\makebox(0,0){$\bullet$}}
	\put(177,-55){\makebox(0,0){$4|123$}} \put(177,-300){\makebox
		(0,0){$123|4$}} \put(55,-180) {\makebox(0,0){$12|34$}}
	\put (300,-136){\makebox(0,0){$34|12$}}
	\put(150,-180){\makebox(0,0){$2|134$}} \put(300,-195)
	{\makebox(0,0){$3|124$}} \put(210,-152){\makebox
		(0,0){$234|1$}} \put(210,-210){\makebox(0,0){$23|41$}}
	\put (420,-180){\makebox(0,0){$1|234$}}
\end{picture}\vspace{2.6in}

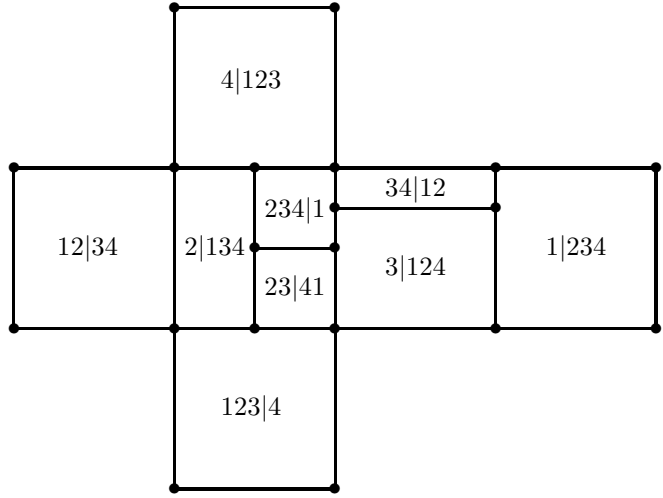
Figure 3: The facets of $K_{5}$ labeled with their minimal subdivision $2$%
-cells in $P_4$.
\end{center}

\medskip

\begin{theorem}
\label{matching}Let $F\times G\subset K_{n+1}\times K_{n+1}$ be an MP. Then $%
F_{\min}\times G_{\min}\subset P_{n}\times P_{n}$ is a CP and $F\times
G=\theta\left( F_{\min}\right) \times\theta\left( G_{\min}\right) .$
Consequently, the diagonals $\Delta_{K}^{\prime}$ and $\Delta_{K}$ agree.
\end{theorem}

\begin{proof}
Let $\sigma=\max F;$ then $F_{\max}=\overleftarrow{\sigma}_{1}|\cdots |%
\overleftarrow{\sigma}_{p}$ for some $p$ and $F_{\min}=R_{\mathbf{M}}\left(
F_{\max}\right) $ for some $\mathbf{M}$ by Proposition \ref{minimal}. Let $%
\beta=\overrightarrow{\sigma}_{q}|\cdots|\overrightarrow{\sigma}_{1}$ and
consider the SCP $F_{\max}\times\beta.$ Since $\sigma$ is an associahedral
vertex and $\min\beta\leq\min G_{\min}$ the hypotheses of Proposition \ref%
{left-action} is satisfied; hence $G_{\min}=L_{\mathbf{N}}(\beta)$ for some $%
\mathbf{N.}$ Therefore $F_{\min}\times G_{\min}=R_{\mathbf{M}}\left(
F_{\max}\right) \times L_{\mathbf{N}}(\beta)$ is a CP and $F\times
G=\theta\left( F_{\min}\right) \times\theta\left( G_{\min}\right) $.
\end{proof}

\begin{example}
Consider the diagonal component 
\begin{equation*}
F\times G=(\bullet \bullet \bullet )\bullet \bullet \times \bullet (\bullet
\bullet (\bullet \bullet ))
\end{equation*}%
of $\Delta _{K}^{\prime }(K_{5}).$ Then $F=21|43\cup 421|3$ is an
associahedral $2$-cell$,$ $\sigma =\max F=4|2|1|3$ is an associahedral vertex%
$,$ 
\begin{equation*}
F_{\max }=\overleftarrow{\sigma }_{1}|\overleftarrow{\sigma }_{2}=421|3,%
\text{ \ and \ }F_{\min }=21|43=R_{\left\{ 4\right\} }\left( 421|3\right) .
\end{equation*}%
Furthermore, 
\begin{equation*}
\beta =\overrightarrow{\sigma }_{3}|\overrightarrow{\sigma }_{2}|%
\overrightarrow{\sigma }_{1}=4|2|13,\min \beta _{1}=4|2|1|3=\max F,\text{ and%
}
\end{equation*}
\begin{equation*}
G_{\min }=L_{\left\{ 3\right\} }(4|2|13)=4|23|1.
\end{equation*}%
Thus $F\times G=\theta \left( 21|43\right) \times \theta \left(
4|23|1\right) .$\medskip
\end{example}


\noindent\textbf{Addendum.} After this paper was written, B. Delcroix-Oger,
G. Laplante-Anfossi, V. Pilaud, and K. Stoeckl proved that $\Delta_{P}$ can
be recovered from $\Delta_{P}^{\prime}$ by an appropriate choice of chambers
in the fundamental hyperplane arrangements of the permutahedra (see \cite%
{DLPS}). The fact that all known diagonals on the associahedra agree (up to
mirror symmetry) follows immediately.

\end{document}